\documentclass[12pt]{amsart}

\usepackage{amsfonts,amsthm,amsmath,amssymb,amscd,mathrsfs,tikz}
\usepackage{graphics}
\usepackage{indentfirst}
\usepackage{cite}
\usepackage{latexsym}
\usepackage[dvips]{epsfig}
\usepackage{enumitem}
\usepackage{color}
\usepackage{amssymb,amsmath,bm}
\usepackage{calc}
\newtheorem{theorem}{Theorem}[section]
\newtheorem{remark}{Remark}[section]
\newtheorem{definition}{Definition}[section]
\newtheorem{lemma}[theorem]{Lemma}

\newtheorem{cor}[theorem]{Corollary}

\newcommand{\bt}{\begin{theorem}}
\newcommand{\bl}{\begin{lemma}}
\newcommand{\el}{\end{lemma}}
\newcommand{\et}{\end{theorem}}

\newcommand{\ka}{\kappa}

\newcommand{\bn}{\begin{eqnarray}}
\newcommand{\en}{\end{eqnarray}}
\newcommand{\bnn}{\begin{eqnarray*}}
\newcommand{\enn}{\end{eqnarray*}}

\newcommand{\ba}{\begin{aligned}}
\newcommand{\ea}{\end{aligned}}
\newcommand{\be}{\begin{equation}}
\newcommand{\ee}{\end{equation}}

\newcommand{\R}{\mathbb{R}}

\usepackage[T1]{fontenc}
\usepackage{lmodern}
\usepackage{amsmath,amssymb,mathtools}
\usepackage{microtype}
\usepackage[hidelinks]{hyperref}
\newtheorem{proposition}{Proposition}[section]
\numberwithin{equation}{section}

\DeclareMathOperator{\diver}{div}
\newcommand{\norm}[2]{\left\|#1\right\|_{#2}}
\newcommand{\avg}[2]{\left(#1\right)_{#2}}
\newcommand{\ep}{\varepsilon}
\newcommand{\ut}{\widehat{\Bu}}
\newcommand{\pt}{\widehat p}
\newcommand{\wt}{\widehat{\Bw}}
\newcommand{\cR}{\mathcal R}

\allowdisplaybreaks[1]
\newcommand{\Bf}{{\boldsymbol{f}}}
\newcommand{\Bv}{{\boldsymbol{v}}}

\newcommand{\Bn}{{\boldsymbol{n}}}
\newcommand{\Bt}{{\boldsymbol{\tau}}}
\newcommand{\Bu}{{\boldsymbol{u}}}

\newcommand{\Be}{{\boldsymbol{e}}}
\newcommand{\BF}{{\boldsymbol{F}}}

\newcommand{\Bvphi}{{\boldsymbol{\varphi}}}

\newcommand{\BU}{\boldsymbol{U}}

\newcommand{\BPsi}{{\boldsymbol{\Psi}}}

\newcommand{\Bw}{{\boldsymbol{w}}}
\newcommand{\BD}{{\boldsymbol{D}}}
\newcommand{\Ba}{{\boldsymbol{a}}}

\newcommand{\BZ}{{\boldsymbol{Z}}}
\newcommand{\BPhi}{{\boldsymbol{\Phi}}}
\newcommand{\Bpsi}{{\boldsymbol{\psi}}}
\begin{document}

\title[Pointwise decay of steady flows with Navier slip]{Pointwise decay of steady Navier--Stokes flows with Navier-slip boundary condition in two-dimensional channels}

  \author{Han Li}
\address{School of Mathematical Sciences, Soochow University, No. 1 Shizi Street, Suzhou, China}
\email{hli3@stu.suda.edu.cn}

		\author{Jingwen Han}
		\address{School of Mathematics and Statistics, Anhui Normal University, Wuhu, China}
		\email{mathjwh95@ahnu.edu.cn}



\begin{abstract}
We establish a pointwise decay estimate for steady incompressible
Navier--Stokes flows in two-dimensional channels with  positive
Navier friction coefficient. Specifically, for the arbitrary-flux solutions constructed by Sha, Wang, and Xie \cite{SWX22f}, under suitable outlet
conditions and weighted geometric bounds on the channel walls, we show
\[
 |\Bu(x_1,x_2)|\le \frac{C}{f(x_1)},
\]
where $f$ is the channel width.  Our decay
rate is optimal in the cross-sectional supremum norm when the prescribed
flux is nonzero.    The main difficulty is that rescaling by
the local width changes the Navier friction coefficient into a
potentially large parameter. The key technique is to establish a local Stokes boundary estimate uniform in this
parameter.
\end{abstract}
\maketitle
\section{Introduction}

We consider steady incompressible viscous flow through a two-dimensional
channel whose width may increase at infinity. The domain is
\begin{equation}
 \Omega=\{x=(x_1,x_2)\in\R^2:
              f_1(x_1)<x_2<f_2(x_1)\},
 \label{intro:domain}
\end{equation}
where $f_1,f_2\in C^3(\R)$ and $f_1(x_1)<f_2(x_1)$ for every $x_1\in\R$. We write
\begin{equation}
 \begin{gathered}
 f(x_1)=f_2(x_1)-f_1(x_1),\qquad
 \bar f(x_1)=\frac{f_1(x_1)+f_2(x_1)}{2},\\
 \Sigma(t)=\{(t,x_2):f_1(t)<x_2<f_2(t)\}.
 \end{gathered}
 \label{intro:width}
\end{equation}
The velocity $\Bu=(u_1,u_2)$ and pressure $p$ satisfy
\begin{equation} \label{intro:NS}
 \begin{cases}
 -\Delta \Bu+\Bu\cdot\nabla \Bu+\nabla p=0 &\text{in }\Omega,\\
 \diver \Bu=0 &\text{in }\Omega,
 \end{cases}
\end{equation}
with the Navier slip boundary condition
\begin{equation} \label{intro:Navier}
 \Bu\cdot \Bn=0,\qquad [\BD(\Bu)\Bn]\cdot\Bt+\alpha \Bu\cdot\Bt=0
       \quad\text{on }\partial\Omega.
\end{equation}
Here $\BD(\Bu)=(\nabla \Bu+\nabla \Bu^{\mathsf T})/2$ is the symmetric velocity gradient, $\Bn$ is the outward
unit normal vector, $\Bt$ is a unit tangent vector, and $\alpha>0$ is a fixed
constant.  The flow is
prescribed by its flux,
\begin{equation}
 \int_{f_1(t)}^{f_2(t)}u_1(t,x_2)\,dx_2=\Phi,
 \label{intro:flux}
\end{equation}
which is independent of $t$ by incompressibility and impermeability.
We take $\Phi\ge0$, with no restriction on its size. Our purpose is
to determine how the velocity decays as the cross-section widens.

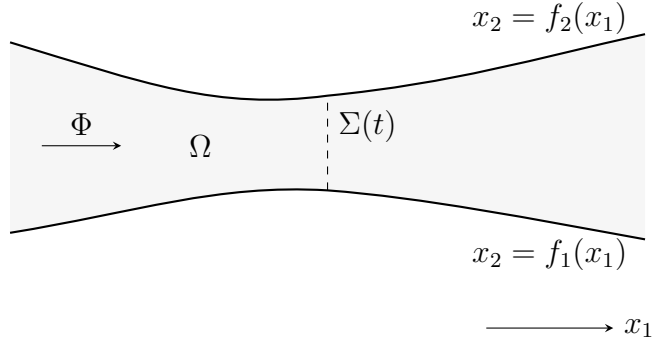
\begin{figure}[htbp]
\centering
\begin{tikzpicture}[x=1.05cm,y=0.72cm,>=stealth]
 \fill[gray!7] (-4,1.9) .. controls (-2.3,1.15) and (-1.4,0.65) .. (0,0.92)
 .. controls (1.6,1.15) and (2.6,1.6) .. (4,2.05)
 -- (4,-1.72) .. controls (2.6,-1.35) and (1.2,-0.95) .. (0,-0.82)
 .. controls (-1.4,-0.7) and (-2.4,-1.22) .. (-4,-1.6) -- cycle;
 \draw[thick] (-4,1.9) .. controls (-2.3,1.15) and (-1.4,0.65) .. (0,0.92)
 .. controls (1.6,1.15) and (2.6,1.6) .. (4,2.05);
 \draw[thick] (-4,-1.6) .. controls (-2.4,-1.22) and (-1.4,-0.7) .. (0,-0.82)
 .. controls (1.2,-0.95) and (2.6,-1.35) .. (4,-1.72);
 \draw[dashed] (0,-0.82)--(0,0.92);
 \node[right] at (0,0.37) {$\Sigma(t)$};
 \node at (-1.6,0) {$\Omega$};
 \node[above] at (2.8,1.85) {$x_2=f_2(x_1)$};
 \node[below] at (2.8,-1.5) {$x_2=f_1(x_1)$};
 \draw[->] (-3.6,0)--(-2.6,0) node[midway,above] {$\Phi$};
 \draw[->] (2,-3.35)--(3.6,-3.35) node[right] {$x_1$};
\end{tikzpicture}
\caption{The channel $\Omega$ and a cross-section $\Sigma(t)$.}
\label{intro:figure}
\end{figure}

The relation between prescribed flux and far-field behaviour is central
to the Leray problem for steady flow in unbounded channels. For straight
outlets, the classical problem seeks solutions approaching the
Poiseuille profiles with the prescribed flux. Foundational results for
the no-slip problem were obtained by Amick \cite{Amick77,Amick78} and
Ladyzhenskaya and Solonnikov \cite{LadySol80}; see \cite{Ga} for a
systematic reference. Quantitative spatial decay is also central to the
entry-flow problem: Horgan and Wheeler \cite{HorganWheeler78} derived
energy-based decay estimates for the development of flow in cylindrical
pipes. These studies connect control of the energy with the approach to
a fully developed flow.

When the outlet width varies, its geometry also enters the description
of the far field. Steady flows in channels of various types were studied
by Amick and Fraenkel \cite{AmickFraenkel80}, while Pileckas
\cite{Pileckas97} treated strong solutions in domains with exits to
infinity. More detailed asymptotic descriptions are available for
specific geometries, such as the paraboloidal outlets considered by
Nazarov and Pileckas \cite{NazarovPileckas98}. For a widening channel,
the flux constraint suggests a velocity scale of order $f(x_1)^{-1}$.
A pointwise bound at this scale describes the size of the flow without
requiring convergence to a prescribed asymptotic profile.

The direct antecedent in this direction is Li and Sha \cite{LS}.
For two-dimensional no-slip channels, they constructed arbitrary-flux
solutions and proved
$|\Bu(x)|\le C/f(x_1)$ \cite[Theorem~1.2]{LS}  under additional outlet assumptions. Their argument combines
local energy estimates on strips of width-dependent length with
rescaling and Stokes regularity. The resulting bound is uniform across
the cross-section and gives decay to zero along widening outlets.
It therefore provides an asymptotic conclusion for arbitrary flux,
without requiring small-flux uniqueness or identifying a limiting
velocity profile.

The Navier slip boundary condition couples tangential stress to tangential velocity
while retaining impermeability. Its role as an effective boundary law
for rough surfaces was studied by J\"ager and Mikeli\'c \cite{JM}.
Analytically, replacing the no-slip condition changes the boundary
estimates needed to convert energy control into pointwise bounds.
Regularity theory for Stokes systems with slip-type conditions was
developed by Beir\~ao da Veiga \cite{Beirao04} and Amrouche and Rejaiba
\cite{AmroucheRejaiba14}. Acevedo Tapia, Amrouche, Conca, and Ghosh
\cite{TACG} further studied Stokes and Navier--Stokes problems with
Navier slip boundary conditions, including estimates uniform in a constant friction
parameter. These results provide the boundary regularity framework
for the present analysis.

For steady channel flows, the distinction between zero and positive
friction is important. In the frictionless setting, Mucha
\cite{MuchaPipe03} obtained solutions in infinite pipes without a
small-flux restriction; his subsequent asymptotic analysis
\cite{Mucha03} quantified exponential convergence to constant end
states, with different inlet and outlet decay rates for nonzero flux.
Sha, Wang, and Xie \cite{SWXfull} also treated arbitrary-flux flows
with full slip. For positive friction coefficient, Li, Pan, and Yang
\cite{LPY2D,LPY3D} studied Leray's problem and convergence to Poiseuille
flows under flux restrictions; in their two-dimensional result, the
admissible flux range expands as friction decreases. Wang and Xie
\cite{WX3} established uniform structural stability in a straight pipe
for axisymmetric solutions. These results concern different geometric
or structural settings from the widening channels considered here.

Our starting point is the work of Sha, Wang, and Xie \cite{SWX22f}
on two-dimensional channels with positive friction and possibly
unbounded width. They constructed solutions with arbitrary flux and,
under additional outlet assumptions, obtained the local energy decay
estimates in Propositions~5.6 and~5.7. We use these estimates for their
constructed solutions, without invoking small-flux uniqueness.
The remaining task is to pass from integral control on local strips
to a pointwise bound uniform across each cross-section. Interior
regularity provides this passage away from the walls; the main issue
is to retain the same inverse-width rate up to a curved Navier  slip boundary.

The difficulty is that rescaling by the local width $r$ changes the
friction coefficient to $\lambda=\alpha r$, which may become arbitrarily
large. We adapt the constant-friction argument of
\cite[Theorem~4.5]{TACG} to a local two-dimensional Stokes estimate
allowing nonzero divergence and nonhomogeneous tangential boundary
data, with a constant independent of $\lambda$. Weighted bounds on
the wall derivatives control the geometric terms introduced by
flattening and the correction to the friction coefficient.
Interpolation and absorption then yield a uniform bound for the
rescaled velocity. This extends the inverse-width pointwise decay
mechanism of \cite{LS} to the arbitrary-flux solutions of
\cite{SWX22f}, including at the channel walls.

\subsection{Notation and weak solutions}
For $a<b$ and $T>0$, set
\[
 \Omega_{a,b}=\Omega\cap\{a<x_1<b\},\qquad
 \Omega_T=\Omega_{-T,T},\qquad
 \Gamma_{a,b}=\partial\Omega\cap\{a<x_1<b\}.
\]
We retain the notation for the energy and test spaces used in
\cite{SWX22f}. Define
\begin{align*}
 \mathcal H(\Omega_{a,b})
 &=\{\Bv\in H^1(\Omega_{a,b};\R^2):
       \Bv=0\text{ on }\Sigma(a)\cup\Sigma(b),\
       \Bv\cdot \Bn=0\text{ on }\Gamma_{a,b}\},\\
 \mathcal H_\sigma(\Omega_{a,b})
 &=\{\Bv\in\mathcal H(\Omega_{a,b}):\diver \Bv=0\},
\end{align*}
and
\begin{equation}
 H_\sigma(\Omega)=
 \left\{\Bu:\begin{array}{l}
 \Bu|_{\Omega_T}\in H^1(\Omega_T;\R^2)\text{ for every }T>0,\\
 \diver \Bu=0\text{ in }\Omega,\quad \Bu\cdot \Bn=0\text{ on }\partial\Omega
 \end{array}\right\}.
 \label{intro:energy-space}
\end{equation}
The local $H^1$ requirement in \eqref{intro:energy-space} includes
the physical boundary. Test fields in
$\mathcal H_\sigma(\Omega_T)$ are extended by zero outside $\Omega_T$.
For a bounded domain $D$, write
\[
 L^2_0(D)=\left\{q\in L^2(D):\int_Dq\,dx=0\right\}.
\]
We also use the zero-flux space
\begin{equation}
 H^1_*(\Omega_{a,b})=
 \left\{\Bv\in H^1(\Omega_{a,b};\R^2):
  \int_{f_1(t)}^{f_2(t)}v_1(t,x_2)\,dx_2=0
                \text{ for a.e. }t\in(a,b)\right\}.
 \label{intro:zero-flux-space}
\end{equation}

\begin{definition}[Weak solution]\label{wea:defwo}
A field $\Bu\in H_\sigma(\Omega)$ is a weak solution of
\eqref{intro:NS}--\eqref{intro:flux} if it has flux $\Phi$ and satisfies
\begin{equation} \label{intro:weak}
 \int_{\Omega_T}\bigl(2\BD(\Bu):\BD(\Bvphi)
                     +(\Bu\cdot\nabla \Bu)\cdot\Bvphi\bigr)\,dx
       +2\alpha\int_{\Gamma_{-T,T}}\Bu\cdot\Bvphi\,ds=0,
\end{equation}
for every $T>0$ and every
$\Bvphi\in\mathcal H_\sigma(\Omega_T)$.
\end{definition}

For a weak solution in Definition~\ref{wea:defwo}, pressure recovery on bounded channel segments gives a pressure $p$,
unique up to an additive constant, such that
$p|_{\Omega_T}\in L^2(\Omega_T)$ for every $T>0$ and
\begin{equation}
 \int_{\Omega_T}\bigl(2\BD(\Bu):\BD(\Bvphi)
                     +(\Bu\cdot\nabla\Bu)\cdot\Bvphi
                     -p\,\diver\Bvphi\bigr)\,dx
 +2\alpha\int_{\Gamma_{-T,T}}\Bu\cdot\Bvphi\,ds=0,
 \label{intro:full-weak}
\end{equation}
for every $\Bvphi\in\mathcal H(\Omega_T)$.
One may refer to
\cite[Theorem~III.5.3]{Ga}. 

\subsection{Geometric assumptions and the main result}
Throughout the paper, we assume
\begin{equation}
 d_*:=\inf_{t\in\R}f(t)>0,\qquad
 \beta:=\max_{i=1,2}\norm{f_i'}{L^\infty(\R)}<\infty.
 \label{intro:basic-geometry}
\end{equation}
For the boundary estimates, we assume
\begin{equation}
 K_2:=\max_{i=1,2}\sup_{t\in\R}f(t)^{3/2}|f_i''(t)|<\infty,
 \qquad
 K_3:=\max_{i=1,2}\sup_{t\in\R}f(t)^2|f_i'''(t)|<\infty.
 \label{intro:K23}
\end{equation}
Since $f\ge d_*$, we have
$\max_i\sup_{\R}f|f_i''|\le d_*^{-1/2}K_2$.
Consequently, \cite[Theorem~1.2]{SWX22f} gives a weak solution of
\eqref{intro:NS}--\eqref{intro:Navier} for every $\Phi\ge0$ and fixed
$\alpha>0$. 

To state the outlet assumptions, let
$F_+(t)=f(t)$ and $F_-(t)=f(-t)$ for $t\ge0$.
For each sign $\sigma\in\{+,-\}$, assume that either
\begin{equation}
 \int_0^\infty F_\sigma(t)^{-3}\,dt=\infty,
 \qquad \lim_{t\to\infty}F_\sigma'(t)=0,
 \label{intro:outlet-I}
\end{equation}
or
\begin{equation}
 \int_0^\infty F_\sigma(t)^{-3}\,dt<\infty,
 \qquad
 \lim_{t\to\infty}
 \frac{\sup_{s\ge t}|F_\sigma'(s)|}
      {\left(\int_t^\infty F_\sigma(s)^{-3}\,ds\right)^{1/2}}=0.
 \label{intro:outlet-II}
\end{equation}
Under the outlet assumptions \eqref{intro:outlet-I} or \eqref{intro:outlet-II},
the following estimate from \cite[Propositions~5.6 and~5.7]{SWX22f} 
\begin{equation}
 \norm{\nabla \Bu}{L^2(D_t^\sigma)}^2
       +\alpha\norm{\Bu}{L^2(\partial\Omega\cap\partial D_t^\sigma)}^2
       \le\frac{C}{F_\sigma(t)^2},\qquad t\ge0
 \label{intro:local-energy}
\end{equation}
holds,
where the constant in \eqref{intro:local-energy} depends on $\alpha$, $\Phi$,
and the channel, but not on $t$. Here we set 
\[
 D_t^+=\Omega_{t-\beta_*f(t),t},\qquad
 D_t^-=\Omega_{-t,-t+\beta_*f(-t)},
\]
 with $\beta_*=[4(1+\beta)]^{-1}$.
This estimate is available without
any smallness assumption on $\Phi$. Our main result converts this estimate into
a pointwise bound valid up to the walls.

\begin{theorem}[Pointwise decay]
\label{intro:main}
Let $\Omega$ be given by \eqref{intro:domain}, with
$f_1,f_2\in C^3(\R)$ satisfying \eqref{intro:basic-geometry} and
\eqref{intro:K23}. Suppose that each outlet satisfies
\eqref{intro:outlet-I} or \eqref{intro:outlet-II}.
Fix $\alpha>0$ and any flux $\Phi\ge0$, and let $\Bu$ be a weak
solution constructed in \cite[Theorem~1.2]{SWX22f}. Then one has 
\begin{equation}
 |\Bu(x_1,x_2)|\le\frac{C}{f(x_1)}
       \qquad\text{for all }(x_1,x_2)\in\overline\Omega,
 \label{intro:pointwise}
\end{equation}
where $C$ depends on $\alpha$, $\Phi$, and $\Omega$,
including $d_*,\beta,K_2,K_3$ and the outlet bounds, but is independent
of $x$.
In particular, if $f(t)\to\infty$ as $t\to+\infty$ or $t\to-\infty$,
then
\begin{equation}
 \sup_{f_1(t)\le x_2\le f_2(t)}|\Bu(t,x_2)|\longrightarrow0
 \label{intro:uniform-decay}
\end{equation}
at that outlet.
\end{theorem}

\begin{remark}[Scope of the conclusion]
The conclusion applies to the solutions constructed in \cite{SWX22f},
for which the local energy estimate is available.
The constant may depend on $\alpha$.
On outlets of bounded width, the estimate gives a uniform velocity bound.
\end{remark}

\begin{remark}[Sharpness of the decay rate]
For $\Phi>0$, the flux constraint and \eqref{intro:pointwise} give
\begin{equation}
 \frac{\Phi}{f(t)}
 \le \sup_{f_1(t)\le x_2\le f_2(t)}|\Bu(t,x_2)|
 \le \frac{C}{f(t)}.
 \label{intro:sharp-scale}
\end{equation}
Thus the order $f(t)^{-1}$ is sharp in the cross-sectional supremum
norm. This does not assert a matching lower bound at each point,
or convergence of the rescaled velocity to a prescribed profile.
\end{remark}

\begin{remark}[Role of the weighted wall bounds]
The weights in \eqref{intro:K23} are sufficient for our argument;
their optimality is not claimed. After rescaling by the local width,
\[
 |(f_i^r)''|\le CK_2r^{-1/2},\qquad
 |(f_i^r)'''|\le CK_3.
\]
As shown in Step~1 of Proposition~\ref{prop:boundary-decay}, rotation
to tangent coordinates  gives
\[
 \|\zeta'\|_{L^\infty(-2\ell,2\ell)}\le C\ell K_2r^{-1/2},\qquad
 \|\zeta'''\|_{L^\infty(-2\ell,2\ell)}
 \le C(K_3+K_2^2r^{-1}).
\]
Since $r\ge d_*>0$, a fixed small $\ell$ makes the slope uniformly
small while the higher coefficient derivatives remain bounded.
\end{remark}

\begin{remark}[Geometry and a class of widening channels]
The bounds \eqref{intro:K23} are sufficient conditions for the
boundary argument and are stronger than $\sup f|f_i''|<\infty$,
the second-derivative condition used in \cite{SWX22f}. For example, set
\[
 f(t)=(1+t^2)^{a/2},\qquad f_1(t)=-\tfrac12 f(t),\qquad
 f_2(t)=\tfrac12 f(t),\qquad 0<a<\tfrac35.
\]
As $|t|\to\infty$,
\[
 f(t)^{3/2}|f_i''(t)|=O(|t|^{5a/2-2}),\qquad
 f(t)^2|f_i'''(t)|=O(|t|^{3a-3}),
\]
so both $K_2$ and $K_3$ are finite. For $0<a\le1/3$,
\eqref{intro:outlet-I} holds. For $1/3<a<3/5$,
\[
 \int_t^\infty F_\sigma(s)^{-3}\,ds\asymp t^{1-3a},\qquad
 \sup_{s\ge t}|F_\sigma'(s)|\asymp t^{a-1},
\]
and the quotient in \eqref{intro:outlet-II} is therefore of order
$t^{(5a-3)/2}\to0$. Theorem~\ref{intro:main} yields
$|\Bu(x)|\le C(1+x_1^2)^{-a/2}$. Within this family, the restriction
$a<3/5$ comes from the outlet conditions: at $a=3/5$ the quotient
has a positive limit, although $K_2$ and $K_3$ remain finite.
\end{remark}

The rest of this paper is organized as follows. Section~\ref{sec:preliminaries} reviews the local Stokes estimate
and derives the rescaled $H^1$ bound from the energy estimate of
\cite{SWX22f}. Section~\ref{sec:pointwise} proves
Theorem~\ref{intro:main} by combining interior regularity estimate with the
uniform boundary estimate. Finally, the parameter-independent Stokes estimate
is proved in Appendix~\ref{app:local-stokes-proof}.

\section{Preliminaries}
\label{sec:preliminaries}

In this section, we first give a local boundary estimate for the   Stokes system with Navier slip boundary condition.
 Then, combining  the  Poincar\'e inequality, we derive the  uniform $H^1$ bound for the solution obtained in \cite[Theorem~1.2]{SWX22f}.
\subsection{A local boundary estimate for the Stokes system}
The following local estimate allows nonzero divergence and is uniform in the friction parameter. Its proof in
Appendix~\ref{app:local-stokes-proof} adapts the tangential
difference-quotient argument of \cite[Theorem~4.5]{TACG}.

For $\rho>0$, set
\[
 B_\rho^+=\{s\in\mathbb R^2:|s|<\rho,\ s_2>0\},
 \qquad I_\rho=\{(s_1,0):|s_1|<\rho\}.
\]
On $I_1$, let $\Bn=-\Be_2$, $\Bt=\Be_1$, and
$\BD(\BU)=(\nabla \BU+\nabla \BU^{\mathsf T})/2$.
Consider the Stokes system with Navier slip type boundary condition
\begin{equation*}\tag{S}\label{eq:local-stokes}
\begin{cases}
 -\Delta \BU+\nabla P=\BF,&\text{in }B_1^+,\\
 \operatorname{div}\BU=G,&\text{in }B_1^+,\\
 \BU\cdot \Bn=0,\quad [\BD(\BU)\Bn]\cdot\Bt+\lambda U_\tau=Y,
 &\text{on }I_1,
\end{cases}
\end{equation*}
where $\lambda\ge0$ is a constant. Here $Y$ is a scalar function.
We consider
the regularity of solutions only in $B_{1/2}^+$, so no boundary
condition is prescribed on the semicircular arc. More precisely, for
every $\Bvphi\in H^1(B_1^+;\mathbb R^2)$ supported away from that arc
and satisfying $\Bvphi\cdot\Bn=0$ on $I_1$, the weak formulation is
\begin{equation}
 \int_{B_1^+}\nabla\BU:\nabla\Bvphi
 -\int_{B_1^+}P\,\operatorname{div}\Bvphi
 +2\lambda\int_{I_1}U_\tau\varphi_\tau
 =\int_{B_1^+}\BF\cdot\Bvphi
 +2\langle Y,\varphi_\tau\rangle,
 \label{eq:local-stokes-weak}
\end{equation}
together with $\operatorname{div}\BU=G$ in distributions.

\begin{lemma}[Local boundary estimate]\label{lem:local-stokes}
Let $(\BU,P)\in H^1(B_1^+;\mathbb R^2)\times L^2(B_1^+)$ be a weak
solution of \eqref{eq:local-stokes}. Suppose that
\[
 \BF\in L^2(B_1^+;\mathbb R^2),\qquad
 G\in H^1(B_1^+),\qquad Y\in H^{1/2}(I_1).
\]
Then $\BU\in H^2(B_{1/2}^+)$ and $P\in H^1(B_{1/2}^+)$, and
\begin{align}\label{eq:local-boundary-estimate}
 \|\BU\|_{H^2(B_{1/2}^+)}+\|\nabla P\|_{L^2(B_{1/2}^+)}
 \le C\bigl(&\|\BU\|_{H^1(B_1^+)}
 +\|P-(P)_{B_1^+}\|_{L^2(B_1^+)}\notag\\
 &+\|\BF\|_{L^2(B_1^+)}+\|G\|_{H^1(B_1^+)}
 +\|Y\|_{H^{1/2}(I_1)}\bigr),
\end{align}
where $(P)_{B_1^+}$ denotes the mean of $P$ and the constant $C$ is independent of
$\lambda$.
\end{lemma}

\begin{proof}
See Appendix~\ref{app:local-stokes-proof}.
\end{proof}

\subsection{Local \texorpdfstring{$H^1$}{H1} bounds and rescaling}
\label{subsec:local-h1-rescaling}

We use the local energy estimate
\eqref{intro:local-energy} to obtain the uniform $H^1$ bound required after rescaling. First, we  recall a Poincar\'e inequality for zero-flux fields.  For the proof, see \cite[Lemma~2.1]{SWX22f}.

\begin{lemma}[Zero-flux Poincar\'e inequality]
\label{lem:zero-flux-poincare}
Let $a_0<b_0$. Suppose that
$\Bv\in H^1(\Omega_{a_0,b_0};\mathbb R^2)$ has zero flux
across almost every cross-section and satisfies
$\Bv\cdot \Bn=0$ on
$\partial\Omega\cap\partial\Omega_{a_0,b_0}$.
Then
\[
 \|\Bv\|_{L^2(\Omega_{a_0,b_0})}
 \le
 C\|f\|_{L^\infty(a_0,b_0)}
 \bigl(1+\|f_2'\|_{L^\infty(a_0,b_0)}\bigr)
 \|\nabla \Bv\|_{L^2(\Omega_{a_0,b_0})}.
\]
\end{lemma}

We retain the notation $F_\sigma(t)$ and $D_t^\sigma$
introduced in Section~1, where $\sigma\in\{+,-\}$.

\begin{cor}[Uniform bounds after rescaling]
\label{cor:rescaled-energy}
Let $\Bu$ be the solution considered in Section~1 and
assume that \eqref{intro:local-energy} holds.
Then, for each $\sigma\in\{+,-\}$,
\begin{equation}
 \|\Bu\|_{L^2(D_t^\sigma)}
 +F_\sigma(t)\|\nabla \Bu\|_{L^2(D_t^\sigma)}
 \le C,
 \qquad t\ge0.
 \label{eq:local-h1-scaled}
\end{equation}
Consequently, let $r>0$ satisfy
$c_0F_\sigma(t)\le r\le C_0F_\sigma(t)$
for fixed positive constants $c_0,C_0$, and set
\[
 D_t^{\sigma,r}
 :=\frac{D_t^\sigma-x_0}{r},
 \qquad
 \Bu^r(z):=r\,\Bu(x_0+rz),
 \qquad x_0\in\mathbb R^2.
\]
Then
\begin{equation}
 \|\Bu^r\|_{H^1(D_t^{\sigma,r})}\le C.
 \label{eq:rescaled-h1-bound}
\end{equation}
The constant $C$ in \eqref{eq:rescaled-h1-bound}
depends only on $\alpha$, $\Phi$, $\Omega$, and $C_0$,
and is independent of $\sigma$, $t$, $r$, and $x_0$.
The dependence on domain $\Omega$ includes the geometric
and outlet parameters entering \eqref{intro:local-energy}.
\end{cor}

\begin{proof}
Fix $\sigma\in\{+,-\}$ and write
$D=D_t^\sigma$ and $R=F_\sigma(t)$.
By the definition of $D$ and the slope bound,
\[
 cR\le f(x_1)\le CR
 \quad\text{on }D,
 \qquad |D|\le CR^2.
\]
Moreover, \eqref{intro:local-energy} gives
\begin{equation}
 \|\nabla \Bu\|_{L^2(D)}\le CR^{-1}.
 \label{eq:local-gradient-input}
\end{equation}

To estimate the velocity itself, choose
$\theta\in C_c^\infty((-1,1))$ with
$\int_{-1}^1\theta(\xi)\,d\xi=1$, and set
\[
 \bar f=\frac{f_1+f_2}{2},
 \qquad
 \psi(x)
 =\Phi\int_{-1}^{\,2(x_2-\bar f(x_1))/f(x_1)}
       \theta(\xi)\,d\xi,
 \qquad
 \Ba=(\partial_2\psi,-\partial_1\psi).
\]
Then $\Ba$ is divergence free, vanishes near the channel
walls, and has flux $\Phi$ across each cross-section.
The geometric assumptions imply
\[
 |\Ba(x)|\le \frac{C|\Phi|}{f(x_1)},
 \qquad
 |\nabla \Ba(x)|\le \frac{C|\Phi|}{f(x_1)^2}.
\]
Here the weighted second-derivative bound and
$f\ge d_*>0$ ensure that
$\sup_{\mathbb R}f|f_i''|<\infty$.
Thus, by width comparability and the area bound,
\begin{equation}
 \|\Ba\|_{L^2(D)}
 +R\|\nabla \Ba\|_{L^2(D)}
 \le C|\Phi|.
 \label{eq:local-carrier-norms}
\end{equation}

The field $\Bu-\Ba$ has zero flux and zero normal trace
on the channel walls. Applying
Lemma~\ref{lem:zero-flux-poincare} and using
\eqref{eq:local-gradient-input}--\eqref{eq:local-carrier-norms},
we obtain
\begin{align*}
 \|\Bu\|_{L^2(D)}
 &\le \|\Bu-\Ba\|_{L^2(D)}+\|\Ba\|_{L^2(D)}\\
 &\le CR\bigl(
       \|\nabla \Bu\|_{L^2(D)}
       +\|\nabla \Ba\|_{L^2(D)}
       \bigr)
       +\|\Ba\|_{L^2(D)}\\
 &\le C.
\end{align*}
Together with \eqref{eq:local-gradient-input}, this proves
\eqref{eq:local-h1-scaled}.

Finally, the two-dimensional change of variables
$x=x_0+rz$ gives
\[
 \|\Bu^r\|_{L^2(D_t^{\sigma,r})}
 =\|\Bu\|_{L^2(D_t^\sigma)},
 \qquad
 \|\nabla \Bu^r\|_{L^2(D_t^{\sigma,r})}
 =r\|\nabla \Bu\|_{L^2(D_t^\sigma)}.
\]
Since $r\asymp F_\sigma(t)$, estimate
\eqref{eq:rescaled-h1-bound} follows.
\end{proof}

%%%%%%%%%%%%%%%%%%%%%%%%%%%%%%%%%%%%%%%%
%%%%%%%%%%%%%%%%%%%%%%%%%%%%%%%%%%%%%%%%%

%%%%%%%%%%%%%%%%%%%%%%%%%%%
%%%%%%%%%%%%%%%%%%%%%%%%%%%
%%%%%%%%%%%%%%%%%%%%%%%%%%%

\section{Proof of the pointwise decay estimate}
\label{sec:pointwise}

In this section, we prove Theorem~\ref{intro:main} under the notation and assumptions
of Section~1. Unless otherwise stated,
$C=C(\alpha,\Phi,\Omega)$ denotes a positive constant
independent of the observation point and the local channel width.

We first consider the right outlet and write $D_t=D_t^+$.
Taking $\sigma=+$ in Corollary~\ref{cor:rescaled-energy},
we obtain
\begin{equation}\label{eq:energy-input}
 \|\Bu\|_{L^2(D_t)}
 +f(t)\|\nabla\Bu\|_{L^2(D_t)}
 \le C,
 \qquad t\ge0.
\end{equation}
The left outlet is treated analogously. Moreover, the definition of $D_t$ and the bound
$\|f'\|_{L^\infty(\mathbb R)}\le2\beta$ imply
\[
 \frac12 f(t)\le f(x_1)\le\frac32 f(t),
 \qquad x\in D_t.
\]

For $\delta\in(0,1/2)$, define
\begin{equation}
 \Omega_{\delta f}
 =\{x\in\Omega:f_1(x_1)+\delta f(x_1)<x_2
                      <f_2(x_1)-\delta f(x_1)\}.
 \label{eq:interior-region}
\end{equation}
Applying the interior argument of Li and Sha
\cite[Proposition~5.2]{LS} with \eqref{eq:energy-input}, we obtain
\begin{equation}
 |\Bu(x)|\le \frac{C_\delta}{f(x_1)},
 \qquad x\in\Omega_{\delta f}.
 \label{eq:interior-decay}
\end{equation}
Indeed, fix an interior point $x^*$ and choose
$r\asymp\delta f(x_1^*)$ so that $B_r(x^*)\Subset\Omega$. With
\[
 \Bu^r(z)=r\Bu(x^*+rz),\qquad p^r(z)=r^2p(x^*+rz),
\]
the rescaled velocity $u^r$ solves the Navier-Stokes equations, that is,
\[
 -\Delta\Bu^r+\Bu^r\cdot\nabla\Bu^r+\nabla p^r=0,
 \qquad \operatorname{div}\Bu^r=0,
\]
on a fixed ball, and \eqref{eq:energy-input} gives a uniform local
$H^1$ bound for $\Bu^r$. In two dimensions,
$H^1\hookrightarrow L^q$ for every finite $q$; applying the local
Stokes estimate first with an exponent below $2$ and then bootstrapping
yields a uniform $H^2$ bound on a smaller ball. Hence
$\|\Bu^r\|_{L^\infty}\le C$ by Sobolev embedding, which is precisely
\eqref{eq:interior-decay} after scaling back. This is the argument of
\cite[Proposition~5.2]{LS}. No boundary condition is used here, therefore  we focus
 below on the channel walls.

\begin{proposition}[Boundary decay]
\label{prop:boundary-decay}
Suppose that the hypotheses of Theorem~\ref{intro:main}
hold, and let $\Bu$ be the weak solution specified there.
Then there exist $\delta_0\in(0,1/2)$ and $C>0$ such that
\begin{equation}
 |\Bu(x)|\le \frac{C}{f(x_1)},
 \qquad
 x\in\Omega\setminus\Omega_{\delta_0 f}.
 \label{eq:boundary-decay}
\end{equation}
The constants $\delta_0$ and $C$ are independent of $x$.
\end{proposition}

\begin{proof}
It suffices to treat a point $x^*=(x_1^*,x_2^*)$ with $x_1^*\ge0$
near the upper wall. The other wall and the left outlet are handled
in the same way, using the corresponding local energy bound.
Thus we assume
\begin{equation}
 0<f_2(x_1^*)-x_2^*\le\delta_0 f(x_1^*).
 \label{eq:point-near-wall}
\end{equation}

\medskip
\noindent\textit{Step 1. Scaling and a uniform boundary chart.}
Choose $t>0$ so that
\[
 x_1^*=t-\tfrac12\beta_*f(t).
\]
Such a $t$ exists because the function on the right is strictly
increasing and tends to infinity. Put
\[
 r=f(x_1^*),\qquad \bar x^*=(x_1^*,f_2(x_1^*)),\qquad
 D_t^r=(D_t-\bar x^*)/r.
\]
Width comparability gives $r\asymp f(t)$. Define
\begin{equation}
 \Bu^r(z)=r\,\Bu(\bar x^*+rz),\qquad
 p^r(z)=r^2p(\bar x^*+rz).
 \label{eq:scaling}
\end{equation}
Let 
\[
 \Gamma_t^r
 :=\left\{\frac{x-\bar{x}^*}{r}:
 x\in\partial\Omega\cap\partial D_t\right\}
\]
denote the portion of $\partial D_t^r$ corresponding
to the channel walls.
Then
\begin{equation}
 \begin{cases}
 -\Delta \Bu^r+\Bu^r\cdot\nabla \Bu^r+\nabla p^r=0,
       &\text{in }D_t^r,\\
 \diver \Bu^r=0,&\text{in }D_t^r,\\
 \Bu^r\cdot \Bn=0,\quad [\BD(\Bu^r)\Bn]\cdot\Bt+\alpha r\,u^r_\tau=0,
       &\text{on }\Gamma_t^r.
 \end{cases}
 \label{eq:scaled-system}
\end{equation}
By Corollary~\ref{cor:rescaled-energy},
with $\sigma=+$ and $x_0=\bar x^*$, we have
\begin{equation}
 \|\Bu^r\|_{H^1(D_t^r)}\le C.
 \label{eq:scaled-h1}
\end{equation}

Choose a rotation $\cR$ that maps the outward normal at $z=0$ to
$-\Be_2$, and set
\begin{equation}
 y=\cR z,\qquad
 \Bu_R(y)=\cR \Bu^r(\cR^{-1}y),\qquad
 p_R(y)=p^r(\cR^{-1}y).
 \label{eq:rotation}
\end{equation}
By rotational invariance, setting
$D_R=\cR D_t^r$ and $\Gamma_R=\cR\Gamma_t^r$,
we obtain
\begin{equation}
 \begin{cases}
 -\Delta \Bu_R+\nabla p_R=\Bf_R, &\text{in }D_R,\\
 \diver \Bu_R=0, &\text{in }D_R,\\
 \Bu_R\cdot \Bn=0, &\text{on }\Gamma_R,\\
 [\BD_y(\Bu_R)\Bn]\cdot\Bt+\alpha r\,\Bu_R\cdot\Bt=0,
      &\text{on }\Gamma_R,
 \end{cases}
 \qquad \Bf_R=-\Bu_R\cdot\nabla \Bu_R.
 \label{eq:rotated-system}
\end{equation}
No additional boundary conditions are prescribed on
the artificial end sections.
By the choice of $\cR$, the rotated image of the upper
wall is locally represented by
\[
 y_2=\zeta(y_1),\qquad \zeta(0)=\zeta'(0)=0,
\]
with the fluid occupying the side $y_2>\zeta(y_1)$.

We next establish a uniform size for this local chart
and the derivative bounds needed below.
The scaled wall functions are
\[
 f_i^r(z_1)=r^{-1}\bigl(f_i(x_1^*+rz_1)-f_2(x_1^*)\bigr).
\]
Writing $X_1=x_1^*+rz_1$, the chain rule gives
\[
 (f_i^r)'=f_i'(X_1),\qquad
 (f_i^r)''=r f_i''(X_1),\qquad
 (f_i^r)'''=r^2 f_i'''(X_1).
\]
By local width comparability, $f(X_1)\asymp r$. Thus
\eqref{intro:basic-geometry}--\eqref{intro:K23} imply
\[
 |(f_i^r)''|\le \frac{K_2r}{f(X_1)^{3/2}},\qquad
 |(f_i^r)'''|\le \frac{K_3r^2}{f(X_1)^2},
\]
and consequently
\begin{equation}
 |(f_i^r)'|\le\beta,\qquad
 |(f_i^r)''|\le Cr^{-1/2},\qquad
 |(f_i^r)'''|\le C.
 \label{eq:scaled-wall-derivatives}
\end{equation}

To make the rotated graph estimates explicit, put
\[
 F=f_2^r,\qquad m=F'(0),\qquad L=(1+m^2)^{1/2}.
\]
Then $F(0)=0$, $|m|\le\beta$, and
\[
 1\le L\le(1+\beta^2)^{1/2}.
\]
The preceding estimates for the scaled wall functions give,
on a fixed interval $(-\rho_0,\rho_0)$,
\[
 \|F''\|_{L^\infty(-\rho_0,\rho_0)}
 \le C_gK_2r^{-1/2},
 \qquad
 \|F'''\|_{L^\infty(-\rho_0,\rho_0)}
 \le C_gK_3,
\]
where $\rho_0>0$ and $C_g$ are independent of $x^*$ and $r$.

Choose
\[
 \cR=\frac1L\begin{pmatrix}-1&-m\\m&-1\end{pmatrix},\qquad
 y_1(z_1)=\frac{-z_1-mF(z_1)}L,\qquad
 y_2(z_1)=\frac{mz_1-F(z_1)}L.
\]
This rotation sends the outward normal $L^{-1}(-m,1)^{\mathsf T}$
to $-\Be_2$. Set
\[
 T=y_1'=-\frac{1+mF'}L,\qquad b=-\frac mL.
\]
Then $T(0)=-L$, $T'=bF''$, $|b|\le1$, and
\[
 |T(z_1)-T(0)|
 \le |b|\|F''\|_{L^\infty(-\rho_0,\rho_0)}|z_1|
 \le C_gK_2d_*^{-1/2}|z_1|.
\]
With $M_0=C_gK_2d_*^{-1/2}$, choose
\[
 \rho=\min\left\{\frac{\rho_0}{2},\frac1{2(1+M_0)}\right\}.
\]
Thus $T\le-1/2$ on $[-\rho,\rho]$. Since $y_1(0)=0$,
\[
 y_1(\rho)\le-\rho/2,\qquad y_1(-\rho)\ge\rho/2.
\]
The map $z_1\mapsto y_1(z_1)$ therefore has a $C^3$ inverse on an
interval containing $(-\rho/2,\rho/2)$. Define
$\zeta(y_1)=y_2(z_1(y_1))$ there. Using $d/dy_1=T^{-1}d/dz_1$,
we obtain
\[
 \zeta'=\frac{m-F'}{LT},\qquad
 \zeta''=\frac{F''}{T^3},\qquad
 \zeta'''=\frac{F'''}{T^4}-\frac{3b(F'')^2}{T^5},
\]
where the right-hand sides are evaluated at $z_1=z_1(y_1)$.
In particular, $\zeta(0)=\zeta'(0)=0$, and
\[
 |\zeta''|\le CK_2r^{-1/2},\qquad
 |\zeta'''|\le C(K_3+K_2^2r^{-1})\le C.
\]

The rotated graph is defined on $(-\rho/2,\rho/2)$,
and the opposite wall and the artificial strip ends
remain a uniform positive distance from the origin.
Choose
\[
 0<\ell<\min\left\{\frac1{16},\frac{\rho}{4}\right\}
\]
sufficiently small that the local chart avoids the
opposite wall and the artificial end sections.
This choice depends only on the fixed geometric
parameters and is independent of $x^*$ and $r$.
The preceding derivative bounds imply
\begin{equation}
 \|\zeta''\|_{L^\infty(-2\ell,2\ell)}
 \le Cr^{-1/2},
 \qquad
 \|\zeta'''\|_{L^\infty(-2\ell,2\ell)}
 \le C.
 \label{eq:graph-higher}
\end{equation}
Writing $s_1=y_1$ for the tangential coordinate used
after flattening, we obtain from $\zeta'(0)=0$ that
\begin{equation}
 \begin{aligned}
 |\zeta'(s_1)|
 &=\left|\int_0^{s_1}\zeta''(\xi)\,d\xi\right|\\
 &\le |s_1|\,
       \|\zeta''\|_{L^\infty(-2\ell,2\ell)}
 \le C|s_1|r^{-1/2},
 \qquad |s_1|<2\ell.
 \end{aligned}
 \label{eq:graph-slope}
\end{equation}
In particular, since $r\ge d_*$,
\[
 \|\zeta'\|_{L^\infty(-2\ell,2\ell)}
 \le 2C\ell r^{-1/2}
 \le 2C\ell d_*^{-1/2}.
\]
The constant $C$ is independent of $\ell$.
Thus $\ell$ can be further reduced, independently
of $x^*$ and $r$, to satisfy the smallness condition
required in Step~5.

Write
\begin{align*}
 \omega&=\{y:|y_1|<\ell,\ \zeta(y_1)<y_2<\zeta(y_1)+\ell\},\\
 \omega'&=\{y:|y_1|<\ell/2,\
                \zeta(y_1)<y_2<\zeta(y_1)+\ell/2\}.
\end{align*}
These domains lie in the rotated strip. The point corresponding to
$x^*$ satisfies $|y^*|=|z^*|\le\delta_0$. By reducing $\delta_0$
in terms of the fixed chart size, we ensure $y^*\in\omega'$.
The scaling and rotation are illustrated in Figure~\ref{fig:scaled-rotation}.

% Requires TikZ (already loaded in the manuscript).
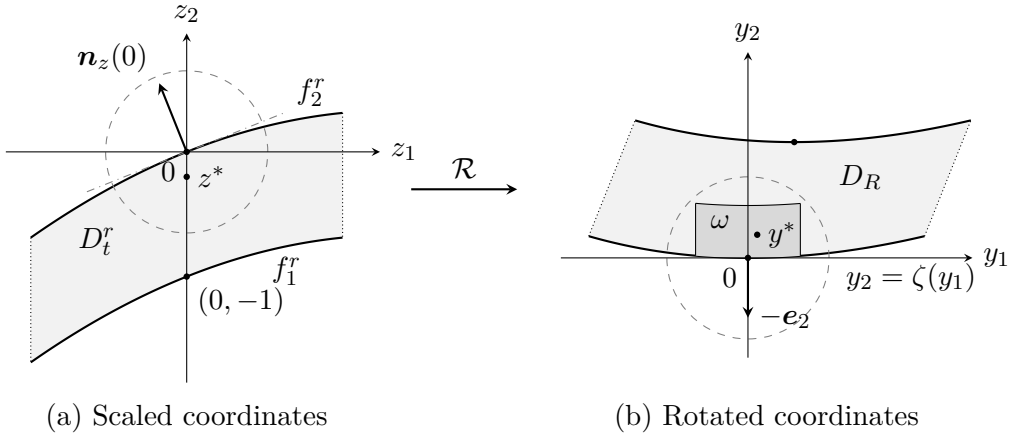
\begin{figure}[htbp]
\centering
\begin{tikzpicture}[x=1.65cm,y=1.65cm,>=stealth,
                    every node/.style={font=\small}]
% Left: a schematic scaled strip; the right panel is its exact rotation.
\begin{scope}[shift={(0,0)}]
 \fill[gray!10] plot[domain=-1.25:1.25,samples=60]
   (\x,{.4*\x-.12*\x*\x})
   -- plot[domain=1.25:-1.25,samples=60]
   (\x,{.4*\x-.12*\x*\x-1}) -- cycle;
 \draw[thick] plot[domain=-1.25:1.25,samples=60]
   (\x,{.4*\x-.12*\x*\x});
 \draw[thick] plot[domain=-1.25:1.25,samples=60]
   (\x,{.4*\x-.12*\x*\x-1});
 \draw[densely dotted] (-1.25,-.6875)--(-1.25,-1.6875)
                          (1.25,.3125)--(1.25,-.6875);
 \draw[gray,dashed] (0,0) circle[radius=.65];
 \draw[gray,dash dot] (-.8,-.32)--(.8,.32);
 \draw[->] (-1.45,0)--(1.55,0) node[right] {$z_1$};
 \draw[->] (0,-1.85)--(0,.95) node[above] {$z_2$};
 \draw[->,thick] (0,0)--(-.22,.55) node[above left] {$\Bn_z(0)$};
 \fill (0,0) circle(1.2pt) node[below left] {$0$};
 \fill (0,-.20) circle(1.2pt) node[right] {$z^*$};
 \fill (0,-1) circle(1.2pt) node[below right] {$(0,-1)$};
 \node at (-.72,-.72) {$D_t^r$};
 \node[above] at (1,.28) {$f_2^r$};
 \node[below] at (.8,-.7568) {$f_1^r$};
 \node at (0,-2.13) {(a) Scaled coordinates};
\end{scope}
\draw[->,thick] (1.8,-.3)--(2.65,-.3)
 node[midway,above] {$\mathcal R$};
% R=[-1,-m;m,-1]/sqrt(1+m^2), m=0.4.
\begin{scope}[shift={(4.5,-.85)}]
 \begin{scope}[rotate=158.1985905]
  \fill[gray!10] plot[domain=-1.25:1.25,samples=60]
    (\x,{.4*\x-.12*\x*\x})
    -- plot[domain=1.25:-1.25,samples=60]
    (\x,{.4*\x-.12*\x*\x-1}) -- cycle;
  \draw[thick] plot[domain=-1.25:1.25,samples=60]
    (\x,{.4*\x-.12*\x*\x});
  \draw[thick] plot[domain=-1.25:1.25,samples=60]
    (\x,{.4*\x-.12*\x*\x-1});
  \draw[densely dotted] (-1.25,-.6875)--(-1.25,-1.6875)
                           (1.25,.3125)--(1.25,-.6875);
 \end{scope}
 % Exact inverse of y_1=(-1.16 z_1+.048 z_1^2)/sqrt(1.16).
 \def\wallheight#1{.12/sqrt(1.16)*
    (-2*sqrt(1.16)*(#1)/(1.16+sqrt(1.3456+.192*sqrt(1.16)*(#1))))^2}
 \fill[gray!30] plot[domain=-.42:.42,samples=35]
    (\x,{\wallheight{\x}})
    -- plot[domain=.42:-.42,samples=35]
    (\x,{\wallheight{\x}+.42}) -- cycle;
 \draw plot[domain=-.42:.42,samples=35]
    (\x,{\wallheight{\x}+.42});
 \draw (-.42,{\wallheight{-.42}})--(-.42,{\wallheight{-.42}+.42})
        (.42,{\wallheight{.42}})--(.42,{\wallheight{.42}+.42});
 \draw[gray,dashed] (0,0) circle[radius=.65];
 \draw[->] (-1.5,0)--(1.8,0) node[right] {$y_1$};
 \draw[->] (0,-.8)--(0,1.65) node[above] {$y_2$};
 \draw[->,thick] (0,0)--(0,-.48) node[right] {$-\Be_2$};
 \fill (0,0) circle(1.2pt) node[below left] {$0$};
 \fill ({.08/sqrt(1.16)},{.2/sqrt(1.16)}) circle(1.2pt)
       node[right] {$y^*$};
 \fill ({.4/sqrt(1.16)},{1/sqrt(1.16)}) circle(1.2pt);
 \node at (-.22,.29) {$\omega$};
 \node at (.9,.65) {$D_R$};
 \node[below right] at (.7,{\wallheight{.7}}) {$y_2=\zeta(y_1)$};
 \node at (.15,-1.28) {(b) Rotated coordinates};
\end{scope}
\end{tikzpicture}
\caption{Rotation of the scaled channel near the upper wall.
The map $\mathcal R$ sends the outward normal at the origin to
$-\Be_2$, so the fluid lies above $y_2=\zeta(y_1)$ and
$\zeta(0)=\zeta'(0)=0$. The shaded local domain is $\omega$;
the dashed circles indicate corresponding neighbourhoods of the origin.
The diagram is schematic.}
\label{fig:scaled-rotation}
\end{figure}

\medskip
\noindent\textit{Step 2. Pressure normalization.}

We estimate the pressure oscillation on $\omega$ using
the divergence-lifting argument in \cite[Remark~5.1]{LS}.
By the bounds for $\zeta'$ and $\zeta''$, the domain $\omega$
is the image of a fixed rectangle under a volume-preserving
$C^2$ diffeomorphism whose first and second derivatives,
as well as the first derivatives of its inverse,
are uniformly bounded.
The  Bogovski\u\i\ map gives a solution to the divergence equations, see \cite{Bo,Ga}.
For every
\[
 \chi\in L^2_0(\omega)
 :=\left\{\chi\in L^2(\omega):
                  \int_\omega\chi=0\right\},
\]
there exists
$\Bvphi_\chi\in H^1_0(\omega;\R^2)$ such that
\begin{equation}
 \diver\Bvphi_\chi=\chi,
 \qquad
 \norm{\nabla\Bvphi_\chi}{L^2(\omega)}
 \le C\norm{\chi}{L^2(\omega)}.
 \label{eq:pressure-lifting}
\end{equation}
The constant depends on the fixed chart size $\ell$
and the geometric bounds, but is independent of $x^*$ and $r$.

Testing
$-\Delta\Bu_R+\nabla p_R=\Bf_R$
with $\Bvphi_\chi$, and using $\int_\omega\chi=0$, yields
\begin{equation}
 \int_\omega
   \bigl(p_R-\avg{p_R}{\omega}\bigr)\chi
 =
 \int_\omega
   \nabla\Bu_R:\nabla\Bvphi_\chi
 -
 \langle\Bf_R,\Bvphi_\chi\rangle.
 \label{eq:pressure-test}
\end{equation}
No boundary term occurs since $\Bvphi_\chi$ has zero trace
on $\partial\omega$.
By \eqref{eq:pressure-lifting} and duality,
\begin{equation}
 \norm{p_R-\avg{p_R}{\omega}}{L^2(\omega)}
 \le C\left(
       \norm{\nabla\Bu_R}{L^2(\omega)}
       +\norm{\Bf_R}{H^{-1}(\omega)}
       \right),
 \label{eq:pressure-dual-bound}
\end{equation}
where $H^{-1}(\omega)$ denotes the dual of
$H^1_0(\omega;\R^2)$ equipped with the gradient norm.
In particular, if $\Bf_R\in L^2(\omega;\R^2)$,
the Poincar\'e inequality gives
\begin{equation}
 \norm{p_R-\avg{p_R}{\omega}}{L^2(\omega)}
 \le C\left(
       \norm{\nabla\Bu_R}{L^2(\omega)}
       +\norm{\Bf_R}{L^2(\omega)}
       \right).
 \label{eq:pressure-force-bound}
\end{equation}

For the present nonlinear force,
the identity $\diver\Bu_R=0$ implies
\[
 \Bf_R=-\Bu_R\cdot\nabla\Bu_R
      =-\diver(\Bu_R\otimes\Bu_R).
\]
Hence, for every $\Bpsi\in H^1_0(\omega;\R^2)$,
\[
 \left|\langle\Bf_R,\Bpsi\rangle\right|
 =
 \left|
 \int_\omega(\Bu_R\otimes\Bu_R):\nabla\Bpsi
 \right|
 \le
 \norm{\Bu_R}{L^4(\omega)}^2
 \norm{\nabla\Bpsi}{L^2(\omega)}.
\]
The two-dimensional Sobolev inequality, with a uniform constant
provided by the graph map, gives
\[
 \norm{\Bf_R}{H^{-1}(\omega)}
 \le C\norm{\Bu_R}{H^1(\omega)}^2.
\]
Combining this bound with \eqref{eq:pressure-dual-bound}
and \eqref{eq:scaled-h1}, we obtain
\begin{equation}
 \norm{p_R-\avg{p_R}{\omega}}{L^2(\omega)}
 \le C\left(
       \norm{\nabla\Bu_R}{L^2(\omega)}
       +\norm{\Bu_R}{H^1(\omega)}^2
       \right)
 \le C.
 \label{eq:pressure-bound}
\end{equation}

We now replace $p_R$ by $p_R-\avg{p_R}{\omega}$
and retain the notation $p_R$.
This normalization leaves the equations and boundary
conditions unchanged. Thus, 
\[
 \avg{p_R}{\omega}=0,
 \qquad
 \norm{p_R}{L^2(\omega)}\le C.
\]

\medskip
\noindent\textit{Step 3. Flattening the boundary.}

Set
\[
 Q=(-\ell,\ell)\times(0,\ell),
 \qquad
 Q'=(-\ell/2,\ell/2)\times(0,\ell/2),
\]
and define
\[
 \BPsi(s)=(s_1,s_2+\zeta(s_1)).
\]
Then $\BPsi(Q)=\omega$, $\BPsi(Q')=\omega'$, and
$\det D\BPsi=1$.
Introduce the flattened coordinates and fields by
\begin{equation}
 s_1=y_1,\qquad
 s_2=y_2-\zeta(y_1),\qquad
 \ut(s)=\Bu_R(\BPsi(s)),\qquad
 \pt(s)=p_R(\BPsi(s)).
 \label{eq:flattening}
\end{equation}
Here $p_R$ denotes the pressure normalized in Step~2.

Set
\[
 I=(-\ell,\ell)\times\{0\},\qquad
 q=\zeta',\qquad
 J=(1+q^2)^{1/2},
\]
and write $\partial_j=\partial/\partial s_j$ below.
The uniform bounds for the coordinate map,
\eqref{eq:scaled-h1}, and the pressure estimate in Step~2 give
\begin{equation}
 M:=\norm{\ut}{H^1(Q)}\le C,\qquad
 P_0:=\norm{\pt}{L^2(Q)}
     =\norm{p_R}{L^2(\omega)}\le C.
 \label{eq:flat-low-norms}
\end{equation}
The derivative relations
$\partial_{y_1}=\partial_1-q\partial_2$ and
$\partial_{y_2}=\partial_2$ give, componentwise,
\begin{align*}
 (\Delta_y \Bu_R)\circ\BPsi
   &=\Delta\ut-2q\partial_{12}\ut+q^2\partial_{22}\ut
                                           -q'\partial_2\ut,\\
 (\nabla_y p_R)\circ\BPsi
   &=\nabla\pt-(q\partial_2\pt,0)^{\mathsf T},\\
 ((\Bu_R\cdot\nabla_y)\Bu_R)\circ\BPsi
   &=\ut\cdot\nabla\ut-q\widehat{u}_1\partial_2\ut.
\end{align*}
Substituting these identities into \eqref{eq:rotated-system} gives the transformed force 
\begin{align}
 \widehat \Bf={}&-\ut\cdot\nabla\ut+q\widehat{u}_1\partial_2\ut
       +(q\partial_2\pt,0)^{\mathsf T}
       -2q\partial_{12}\ut+q^2\partial_{22}\ut-q'\partial_2\ut.
 \label{eq:flat-force}
\end{align}

We next compute the boundary condition. Since the fluid lies
above the graph,
\[
 \Bn=J^{-1}(q,-1)^{\mathsf T},\qquad
 \Bt=J^{-1}(1,q)^{\mathsf T}.
\]
The impermeability condition and the divergence equation imply
\begin{gather}
 \widehat{u}_2=q\widehat{u}_1,\qquad
 \partial_1\widehat{u}_2=q'\widehat{u}_1+q\partial_1\widehat{u}_1
       \quad\text{on }I,\label{eq:trace-identities}\\
 \partial_2\widehat{u}_2=-
 \partial_1\widehat{u}_1+q\partial_2\widehat{u}_1
       \quad\text{in }Q.\label{eq:normal-divergence}
\end{gather}
Substituting these identities into the tangential stress gives
\[
 [\BD_y(\Bu_R)\Bn]\cdot\Bt
 =-\frac{J^2}{2}\partial_2\widehat{u}_1+q\partial_1\widehat{u}_1
        -\frac{(1-q^2)q'}{2J^2}\widehat{u}_1.
\]
As $\Bu_R\cdot\Bt=J\widehat{u}_1$ on $I$, the Navier slip boundary condition becomes
\begin{equation}
 \partial_2\widehat{u}_1-\frac{2\alpha r}{J}\widehat{u}_1
 =\frac{2q}{J^2}\partial_1\widehat{u}_1
       -\frac{(1-q^2)q'}{J^4}\widehat{u}_1.
 \label{eq:flat-navier}
\end{equation}
Thus, the flattened system is
\begin{equation}
 \begin{cases}
 -\Delta\ut+\nabla\pt=\widehat \Bf, &\text{in }Q,\\
 \diver\ut=\widehat g:=q\partial_2\widehat{u}_1, &\text{in }Q,\\
 \widehat{u}_2=q\widehat{u}_1, &\text{on }I,\\
 \displaystyle\partial_2\widehat{u}_1-\frac{2\alpha r}{J}\widehat{u}_1
 =\frac{2q}{J^2}\partial_1\widehat{u}_1
           -\frac{(1-q^2)q'}{J^4}\widehat{u}_1, &\text{on }I,
 \end{cases}
 \label{eq:flattened-system}
\end{equation}
with $\widehat \Bf$ given by \eqref{eq:flat-force}.
For a flat wall, $q\equiv0$ and $J\equiv1$, so the boundary
condition reduces to
$-\tfrac12\partial_2\widehat u_1+\alpha r\widehat u_1=0$,
as expected from the outward normal $-\Be_2$
and the definition of $\BD$.

These identities are justified by local regularity with $\alpha r$
fixed. The Stokes system with  Navier slip boundary conditions
satisfies the complementing boundary condition: after freezing the
coefficients in tangent coordinates, its principal boundary operators
are $U_2=0$ and $\partial_2U_1=0$. The corresponding homogeneous
half-space problem has no nonzero decaying tangential Fourier mode,
as follows by integration by parts. The friction term is of lower
order and does not affect this condition. Thus the local elliptic
regularity theory of \cite{ADN64} applies in two dimensions on the
$C^3$ wall patches; compare
\cite[Theorem~5.10 and Remark~5.11]{TACG}.

For completeness, localize within a larger wall patch using a smooth
cut-off  function $\xi$ with $\partial_n\xi=0$ at the wall. The divergence datum
is $\nabla\xi\cdot\Bu_R\in W^{1,4/3}$, the commutator forces belong
to $L^{4/3}$, and the Navier slip boundary condition is preserved by this cut-off.
Since $\Bu_R\in H^1$ and $p_R\in L^2$ locally up to the wall,
\[
 \Bu_R\cdot\nabla\Bu_R\in L^{4/3}
 \quad\Longrightarrow\quad
 \Bu_R\in W^{2,4/3},\quad p_R\in W^{1,4/3}
\]
on a smaller patch. In two dimensions,
$W^{2,4/3}\hookrightarrow W^{1,4}\hookrightarrow L^\infty$,
so the nonlinear force belongs to $L^2$ there.
Repeating the localization with exponent $2$ yields
$\Bu_R\in H^2$ and $p_R\in H^1$ near the support of the cut-off function
used below. These regularity bounds may depend on $\alpha r$.
The estimates in Step~5 are uniform in $x^*$ and $r$ for fixed $\alpha$.

\medskip
\noindent\textit{Step 4. Localization and a constant friction parameter.}

Choose a smooth cut-off
$\varphi:\overline{\R^2_+}\to[0,1]$
such that $\varphi=1$ on $Q'$,
$\operatorname{supp}\varphi\subset Q\cup I$, and
\[
 \partial_2\varphi=0 \quad\text{on }I,\qquad
 |\nabla\varphi|\le C\ell^{-1},\qquad
 |\nabla^2\varphi|\le C\ell^{-2}.
\]
The support is compact in the closed half-plane and
avoids the lateral and top edges of $Q$.
Define
\begin{equation}
 \Bw=\varphi\ut,\qquad
 \pi=\varphi\pt,\qquad
 \wt=(0,qw_1)^{\mathsf T},\qquad
 \Bv=\Bw-\wt.
 \label{eq:localized-fields}
\end{equation}
Extend these fields by zero to $B_1^+\setminus Q$.
Their supports are contained in
$B_{1/2}\cap\overline{\R^2_+}$.

For the coefficient estimates below, choose
$\vartheta\in C_c^\infty((-2,2))$ such that
$0\le\vartheta\le1$ and $\vartheta=1$ on $[-1,1]$.
Extend $q$ to $\R$ by setting
\[
 q(s_1)=
 \begin{cases}
  \vartheta(s_1/\ell)\zeta'(s_1),
       & |s_1|<2\ell,\\
  0,   & |s_1|\ge2\ell.
 \end{cases}
\]
We retain the notation $q$ for this extension, which agrees
with $\zeta'$ on $[-\ell,\ell]$.
By \eqref{eq:graph-higher}--\eqref{eq:graph-slope},
\[
 \|q\|_{L^\infty(\R)}\le C\ell r^{-1/2},
 \qquad
 \|q'\|_{L^\infty(\R)}\le Cr^{-1/2},
 \qquad
 \|q''\|_{L^\infty(\R)}\le C_\ell,
\]
where $C$ is independent of $\ell$, $x^*$, and $r$,
and $C_\ell$ is independent of $x^*$ and $r$.

We first write the system for $\Bw$, before correcting its normal trace.
The product rules give
\begin{align*}
 -\Delta(\varphi\ut)+\nabla(\varphi\pt)
  &=-\ut\Delta\varphi-2\nabla\varphi\cdot\nabla\ut
                          +\pt\nabla\varphi+\varphi\widehat \Bf,\\
 \diver(\varphi\ut)
  &=\ut\cdot\nabla\varphi+\varphi q\partial_2\widehat{u}_1.
\end{align*}
For compactness, set
\begin{equation}
 \begin{aligned}
 \Bf&=-\ut\Delta\varphi-2\nabla\varphi\cdot\nabla\ut
                          +\pt\nabla\varphi+\varphi\widehat \Bf,\\
 g&=\ut\cdot\nabla\varphi+\varphi q\partial_2\widehat{u}_1,\\
 a&=\frac{q}{J^2},\qquad c=\frac{(1-q^2)q'}{J^4},\qquad
 E_\varphi=(\partial_2\varphi-2a\partial_1\varphi)\widehat{u}_1.
 \end{aligned}
 \label{eq:localized-data}
\end{equation}
The vector $\Bf$ here denotes the localized force; the scalar channel
width $f(x_1)$ always carries its spatial argument when needed below.
Multiplying  the last equation of \eqref{eq:flattened-system}
by $\varphi$ gives
\[
 \partial_2w_1-\frac{2\alpha r}{J}w_1
 =2a(\partial_1w_1-\partial_1\varphi\widehat{u}_1)
                           -cw_1+\partial_2\varphi\widehat{u}_1.
\]
Thus, with $I_1=(-1,1)\times\{0\}$, the localized system is
\begin{equation}
 \begin{cases}
 -\Delta \Bw+\nabla\pi=\Bf, &\text{in }B_1^+,\\
 \diver \Bw=g, &\text{in }B_1^+,\\
 w_2=qw_1, &\text{on }I_1,\\
 \displaystyle\partial_2w_1-\frac{2\alpha r}{J}w_1
       =2a\partial_1w_1-cw_1+E_\varphi, &\text{on }I_1.
 \end{cases}
 \label{eq:w-system}
\end{equation}
All localized data are understood to be zero outside their supports.

Next, the correction $\wt=(0,qw_1)^{\mathsf T}$ satisfies
\[
 \diver\wt=q\partial_2w_1
       =q\partial_2\varphi\widehat{u}_1+q\varphi\partial_2\widehat{u}_1,
 \qquad (\wt)_2=qw_1=w_2\quad\text{on }I_1.
\]
Subtracting this field removes the normal trace. Since $v_1=w_1$,
the equations become
\begin{equation}
 \begin{cases}
 -\Delta \Bv+\nabla\pi=\BF:=\Bf+\Delta\wt, &\text{in }B_1^+,\\
 \diver \Bv=G:=g-\diver\wt, &\text{in }B_1^+,\\
 v_2=0, &\text{on }I_1,\\
 \displaystyle\partial_2v_1-\frac{2\alpha r}{J}v_1
       =2a\partial_1v_1-cv_1+E_\varphi, &\text{on }I_1.
 \end{cases}
 \label{eq:localized-bulk}
\end{equation}
More explicitly,
\begin{equation}
 \begin{aligned}
 G&=\ut\cdot\nabla\varphi-q\widehat{u}_1\partial_2\varphi,\\
 \Delta\wt&=(0,q''w_1+2q'\partial_1w_1+q\Delta w_1)^{\mathsf T}.
 \end{aligned}
 \label{eq:G-and-correction}
\end{equation}
In particular, the derivative in the middle term of $\Delta\wt$
is tangential. A useful expansion of $\BF$ is
\begin{align}
 \BF={}&-\ut\Delta\varphi-2\nabla\varphi\cdot\nabla\ut
       +\pt\nabla\varphi
       +\varphi\bigl(-\ut\cdot\nabla\ut
                    +q\widehat{u}_1\partial_2\ut-q'\partial_2\ut\bigr)
       \notag\\
 &-(q\pt\partial_2\varphi,0)^{\mathsf T}
       +2q\bigl(\partial_1\varphi\partial_2\ut
              +\partial_2\varphi\partial_1\ut
              +\partial_{12}\varphi\ut\bigr)
       \notag\\
 &-q^2\bigl(2\partial_2\varphi\partial_2\ut
                       +\partial_{22}\varphi\ut\bigr)
       -2q\partial_{12}\Bw+q^2\partial_{22}\Bw
       +(q\partial_2\pi,0)^{\mathsf T}
       \notag\\
 &+(0,q''w_1+2q'\partial_1w_1+q\Delta w_1)^{\mathsf T}.
 \label{eq:F-expanded}
\end{align}

Finally, we replace the variable friction coefficient by a constant.
Set
\begin{equation}
 b_r=\frac{\alpha r(J-1)}{J},\qquad
 \frac{\alpha r}{J}=\alpha r-b_r.
 \label{eq:boundary-coefficients}
\end{equation}
The last equation in \eqref{eq:localized-bulk} then reads
\[
 \partial_2v_1-2\alpha rv_1
       =2a\partial_1v_1-cv_1+E_\varphi-2b_rv_1.
\]
With $\Bn_0=-\Be_2$ and $\Bt_0=\Be_1$, the condition $v_2=0$ implies
$[\BD_s(\Bv)\Bn_0]\cdot\Bt_0=-\tfrac12\partial_2v_1$ on $I_1$.
Consequently, the final system takes precisely the form required
in Lemma~\ref{lem:local-stokes}:
\begin{equation}
 \begin{cases}
 -\Delta \Bv+\nabla\pi=\BF, &\text{in }B_1^+,\\
 \diver \Bv=G, &\text{in }B_1^+,\\
 \Bv\cdot \Bn_0=0, &\text{on }I_1,\\
 [\BD_s(\Bv)\Bn_0]\cdot\Bt_0+\alpha r v_1=h,
       &\text{on }I_1,
 \end{cases}
 \label{eq:localized-boundary}
\end{equation}
where
\begin{equation}
 h=-a\partial_1v_1+\tfrac12cv_1-\tfrac12E_\varphi+b_rv_1.
 \label{eq:final-boundary-data}
\end{equation}
The parameter on the left is now the constant $\lambda=\alpha r\ge0$;
$\BF$, $G$, and $h$ are given explicitly above, so each term can be
estimated separately.

The geometric assumptions give
\begin{gather}
 J-1=\frac{q^2}{J+1},\qquad J'=\frac{qq'}{J},\qquad
 \norm{J-1}{L^\infty}+\norm{J'}{L^\infty}\le Cr^{-1},
 \label{eq:J-bounds}\\
 b_r'=\alpha r\frac{J'}{J^2},\qquad
 \norm{b_r}{W^{1,\infty}}\le C\alpha,\qquad
 \norm{c}{W^{1,\infty}}\le C_\ell.
 \label{eq:coefficient-bounds}
\end{gather}
Indeed, the bounds for the extended coefficient $q$ give
\[
 \|b_r\|_{L^\infty(\R)}
 \le C\alpha r\|q\|_{L^\infty(\R)}^2
 \le C\alpha\ell^2,
 \qquad
 \|b_r'\|_{L^\infty(\R)}
 \le C\alpha r
       \|q\|_{L^\infty(\R)}
       \|q'\|_{L^\infty(\R)}
 \le C\alpha\ell.
\]
The derivative of $c$ involves $q''$, which is uniformly
bounded for the fixed chart size $\ell$.
Thus the coefficient bounds are independent of $x^*$ and $r$.

\medskip
\noindent\textit{Step 5. Uniform local estimates.}
All unlabelled norms in this step are taken on $B_1^+$ for the
zero-extended localized fields, on $Q$ for $\ut,\pt$, and on
$I_1=(-1,1)\times\{0\}$ for boundary data. Put
\[
 \eta=\norm{q}{L^\infty},\qquad
 X=\norm{\nabla^2\Bw}{L^2(B_1^+)}
                  +\norm{\nabla\pi}{L^2(B_1^+)}.
\]
The cut-off and coefficient bounds imply
\begin{equation}
 \norm{\Bw}{H^1(B_1^+)}
 +\norm{\Bv}{H^1(B_1^+)}
 \le C_\ell M,
 \qquad
 \norm{G}{H^1(B_1^+)}
 \le C_\ell M.
 \label{eq:lower-localized}
\end{equation}
The constants denoted by $C_\ell$ may also depend on the
fixed geometric bounds, but not on $x^*$ or $r$.

Since $\pi=\varphi\pt$ on $Q$ and is extended by zero
to $B_1^+\setminus Q$, the condition $0\le\varphi\le1$
and \eqref{eq:flat-low-norms} yield
\begin{equation}
 \norm{\pi-\avg{\pi}{B_1^+}}{L^2(B_1^+)}
 \le \norm{\pi}{L^2(B_1^+)}
 = \norm{\varphi\pt}{L^2(Q)}
 \le P_0
 \le C.
 \label{eq:localized-pressure-bound}
\end{equation}

We first estimate the boundary datum. The trace theorem and
multiplication by a $W^{1,\infty}$ coefficient give
\begin{align}
& \norm{b_rv_1}{H^{1/2}(I_1)}
     \le C\norm{b_r}{W^{1,\infty}}\norm{\Bv}{H^1}
       \le C_\ell\alpha M,\label{eq:new-boundary-term}\\
 &\norm{cv_1}{H^{1/2}(I_1)}
       +\norm{E_\varphi}{H^{1/2}(I_1)}
     \le C_\ell M.\label{eq:lower-boundary-terms}
\end{align}
For the term involving a tangential derivative, extend $a$ normally
as a function of $s_1$. Since
\[
 a=\frac{q}{J^2},\qquad
 \|a\|_{L^\infty}\le C\eta,
 \qquad \|a'\|_{L^\infty}\le C_\ell,
\]
only the coefficient multiplying the highest derivative is small;
derivatives of the coefficient contribute lower-order terms. By
$H^1$ trace estimate, it holds that
\begin{align}
 \norm{a\partial_1v_1}{H^{1/2}(I_1)}
 &\le C\norm{a\partial_1v_1}{H^1(B_1^+)}\notag\\
 &\le C\eta\norm{\nabla^2\Bv}{L^2}
       +C_\ell\norm{\Bv}{H^1}.
 \label{eq:tangential-boundary-term}
\end{align}
Consequently,
\begin{equation}
 \norm{h}{H^{1/2}(I_1)}
 \le C\eta\norm{\nabla^2\Bv}{L^2}
             +C_\ell(1+\alpha)M.
 \label{eq:h-estimate}
\end{equation}

For the nonlinear terms, the Gagliardo--Nirenberg interpolation
inequality (see \cite{Fiorenza}), applied after extension from the
fixed half-ball, gives
\[
 \norm{\Bw}{L^\infty}
 \le C\norm{\nabla^2\Bw}{L^2}^{1/2}\norm{\Bw}{L^2}^{1/2}
       +C\norm{\Bw}{L^2}.
\]
Since $\varphi\ut=\Bw$ and $\eta\le1$, it follows that, for every
$\ep>0$,
\begin{align}
 &\norm{\varphi(-\ut\cdot\nabla\ut
                            +q\widehat{u}_1\partial_2\ut)}{L^2}
 \notag\\
 &\qquad\le C\norm{\Bw}{L^\infty}\norm{\nabla\ut}{L^2(Q)}
       \le\ep X+C_\ep M^3+CM^2.
 \label{eq:nonlinear-estimate}
\end{align}
We record separately the remaining contributions in
\eqref{eq:F-expanded},  where the small geometric factor is
used. The terms containing the pressure gradient or second derivatives
of $\Bw$ satisfy
\begin{equation}
 \|q\nabla\pi\|_{L^2}
 +\|q\nabla^2\Bw\|_{L^2}
 +\|q^2\nabla^2\Bw\|_{L^2}
 \le C\eta X.
 \label{eq:F-high-geometric}
\end{equation}
Because $q'$ and $q''$ are uniformly bounded on the fixed chart, while
$\Bw$ and its first derivatives are controlled by
\eqref{eq:lower-localized},
\begin{equation}
 \|q'\nabla\Bw\|_{L^2}
 +\|q''\Bw\|_{L^2}
 \le C_\ell M.
 \label{eq:F-lower-geometric}
\end{equation}
The cut-off terms involving $\ut,\nabla\ut$, and $\pt$ are bounded by
$C_\ell(M+P_0)$. Together with
\eqref{eq:nonlinear-estimate}, these bounds give
\begin{equation}
 \norm{\BF}{L^2}
 \le C_\ell(M+P_0)+C_\ep(M^2+M^3)
                    +\bigl(C\eta+\ep\bigr)X.
 \label{eq:F-estimate}
\end{equation}
The local constant-parameter Stokes estimate of Lemma~\ref{lem:local-stokes}, applied
to \eqref{eq:localized-boundary}, gives
\begin{align}
 &\norm{\nabla^2\Bv}{L^2(B_{1/2}^+)}
       +\norm{\nabla\pi}{L^2(B_{1/2}^+)}\notag\\
 &\quad\le C_*\bigl(
       \norm{\Bv}{H^1(B_1^+)}
       +\norm{\pi-\avg{\pi}{B_1^+}}{L^2(B_1^+)}
       +\norm{\BF}{L^2(B_1^+)}
       +\norm{G}{H^1(B_1^+)}
       +\norm{h}{H^{1/2}(I_1)}\bigr),
 \label{eq:apply-stokes}
\end{align}
where $C_*$ is independent of $\lambda=\alpha r$.
Because $\Bv,\Bw,\pi$ are supported strictly inside $B_{1/2}$ relative
to the half-plane, their derivative norms on $B_{1/2}^+$ agree with
those on $B_1^+$. This support property permits absorption on the
same domain.

The product rule also gives
\begin{equation}
 \norm{\nabla^2\wt}{L^2}
       \le C\eta\norm{\nabla^2\Bw}{L^2}+C_\ell M.
 \label{eq:correction-hessian}
\end{equation}
Using $\Bw=\Bv+\wt$, estimates
\eqref{eq:lower-localized}--\eqref{eq:correction-hessian}
therefore imply
\begin{equation}
 X\le C_{\ell,\alpha,\ep}
           (M+P_0+M^2+M^3)+C_0(\eta+\ep)X.
 \label{eq:absorption}
\end{equation}
The constant $C_0$ is independent of $\ell$, $x^*$, and $r$;
coefficient and cut-off derivatives enter only the lower-order constant.
Choose $\varepsilon>0$ with $C_0\varepsilon\le1/4$, and then fix
$\ell$ sufficiently small, subject to the preceding geometric
restrictions, that $C_0\eta\le1/4$; this is possible because
$\eta\le C_g\ell d_*^{-1/2}$. Since $X$ is finite by local regularity,
absorption in \eqref{eq:absorption}, together with
\eqref{eq:flat-low-norms}, yields
\begin{equation}
 \norm{\nabla^2\Bw}{L^2(B_1^+)}
       +\norm{\nabla\pi}{L^2(B_1^+)}\le C.
 \label{eq:uniform-high-order}
\end{equation}
The constant $C$ is independent of $x^*$ and $r$.

\medskip
\noindent\textit{Step 6. Returning to the original variables.}

Since $\varphi=1$ on $Q'$, we have $\Bw=\ut$ there.
The bounds \eqref{eq:lower-localized} and
\eqref{eq:uniform-high-order}, together with
the embedding $H^2(B_1^+)\hookrightarrow L^\infty(B_1^+)$,
give
\[
 \norm{\ut}{L^\infty(Q')}
 \le \norm{\Bw}{L^\infty(B_1^+)}
 \le C.
\]
Since $y^*\in\omega'=\BPsi(Q')$, the definitions of the
transformed variables and $r=f(x_1^*)$ yield
\begin{equation}
 |\Bu(x^*)|
 =\frac{1}{r}|\Bu^r(z^*)|
 =\frac{1}{r}|\Bu_R(y^*)|
 \le \frac{1}{r}\norm{\ut}{L^\infty(Q')}
 \le \frac{C}{f(x_1^*)}.
 \label{eq:boundary-pointwise}
\end{equation}
The constant is independent of $x^*$ and $r$.
The same argument applies near the lower wall after
choosing the local coordinates so that the fluid lies
above the boundary.
Using the corresponding local energy estimate at the
left outlet gives the same bound there.
This proves the proposition.
\end{proof}

Combining the interior estimate with $\delta=\delta_0$
and Proposition~\ref{prop:boundary-decay}, we obtain
\[
 |\Bu(x)|\le \frac{C}{f(x_1)},
 \qquad x\in\Omega.
\]
The interior bootstrap and the boundary $H^2$ estimates above give a
representative of $\Bu$ that is locally continuous on
$\overline\Omega$, so the pointwise bound extends to the channel
walls by continuity. This finishes the proof of
Theorem~\ref{intro:main}.

\appendix
\section{Proof of the local boundary estimate}
\label{app:local-stokes-proof}

We prove Lemma~\ref{lem:local-stokes}, using the notation of
\eqref{eq:local-stokes}. The argument follows the 
difference-quotient approach in the proof of
\cite[Theorem~4.5]{TACG}. We show the estimate is uniform with respect to the constant $\lambda\ge0$.
All test functions below are supported away from the semicircular boundary;
no boundary condition is required there.

\begin{proof}[Proof of Lemma~\ref{lem:local-stokes}]
Write $\partial_j=\partial/\partial s_j$ and set
\begin{align}
 \mathcal N_0&=\|\BU\|_{H^1(B_1^+)}+\|\BF\|_{L^2(B_1^+)}
       +\|G\|_{H^1(B_1^+)}+\|Y\|_{H^{1/2}(I_1)},
       \label{eq:app-data-zero}\\
 \mathcal N&=\mathcal N_0+
       \|P-(P)_{B_1^+}\|_{L^2(B_1^+)}.
       \label{eq:app-data}
\end{align}

\medskip
\noindent\textit{Step 1. Tangential derivatives of the velocity.}
For $0<|\kappa|<1/32$, put
\[
 \delta_{\kappa} Z(s)=\frac{Z(s+\kappa \Be_1)-Z(s)}{\kappa}.
\]
Choose $3/4\le a<b\le7/8$ and $\chi\in C_c^\infty(B_b)$ with
$0\le\chi\le1$, $\chi=1$ on $B_a$, and
$|\nabla\chi|\le C(b-a)^{-1}$.
The field $\chi^2\delta_{\kappa} \BU$ has zero normal trace on $I_b$
and vanishes near the semicircular boundary of $B_b^+$. Therefore, its divergence $\chi^2\delta_{\kappa} G+2\chi\nabla\chi\cdot\delta_{\kappa} \BU$
 has zero integral. The Bogovski\u{i} map
\cite{Bo,Ga} gives
$\BZ_{\kappa}\in H_0^1(B_b^+;\mathbb R^2)$ such that
\begin{align}
 \operatorname{div}\BZ_{\kappa}
 &=\chi^2\delta_{\kappa} G+2\chi\nabla\chi\cdot\delta_{\kappa} \BU,
       \label{eq:app-div-correction}\\
 \|\nabla \BZ_{\kappa}\|_{L^2(B_b^+)}
 &\le C\left(\|G\|_{H^1(B_1^+)}
       +(b-a)^{-1}\|\BU\|_{H^1(B_1^+)}\right),
       \label{eq:app-correction-bound}
\end{align}
where the constant $C$ is uniform for $b\in(3/4,7/8]$. 

Extend $\BZ_{\kappa}$ and $\chi^2\delta_{\kappa} \BU$ by zero within the upper half-plane
outside $B_b^+$, and set $\BPhi_{\kappa}=\chi^2\delta_{\kappa} \BU-\BZ_{\kappa}$.
Then $\operatorname{div}\BPhi_{\kappa}=0$, its normal trace vanishes,
and $-\delta_{-\kappa}\BPhi_{\kappa}$ is an admissible test function.
Discrete integration by parts in $s_1$ gives
\begin{align}
 &\int_{B_b^+}\nabla\delta_{\kappa} \BU:\nabla\BPhi_{\kappa}\,ds
       +2\lambda\int_{I_b}\delta_{\kappa}U_1\Phi_{\kappa,1}\,ds_1
       \notag\\
 &\hspace{10mm}=-\int_{B_1^+}\BF\cdot\delta_{-\kappa}\BPhi_{\kappa}\,ds
       +2\langle\delta_{\kappa} Y,\Phi_{\kappa,1}\rangle.
       \label{eq:app-difference-identity}
\end{align}
Boundary pairings are taken in
$H^{-1/2}(\R)\times H^{1/2}(\R)$,
using a bounded $H^{1/2}$ extension of $Y$
and the zero extension of the test trace.
They are independent of the chosen extension,
since the test trace and its translates are supported in $I_1$.

Since $\lambda$ is constant, no coefficient difference occurs.
The zero trace of $\BZ_{\kappa}$ gives the exact identity
\begin{equation}
 2\lambda\int_{I_b}\delta_{\kappa} U_1\Phi_{\kappa,1}\,ds_1
 =2\lambda\int_{I_b}\chi^2|\delta_{\kappa} U_1|^2\,ds_1\ge0.
 \label{eq:app-friction-positive}
\end{equation}
The first term expands as
\begin{align}
 \int_{B_b^+}\nabla\delta_{\kappa} \BU:\nabla\BPhi_{\kappa}\,ds
 &=\int_{B_b^+}\chi^2|\nabla\delta_{\kappa} \BU|^2\,ds\notag\\
 &\quad+2\int_{B_b^+}\chi\nabla\delta_{\kappa} \BU:
              (\delta_{\kappa} \BU\otimes\nabla\chi)\,ds\notag\\
 &\quad-\int_{B_b^+}\nabla\delta_{\kappa} \BU:\nabla \BZ_{\kappa}\,ds.
 \label{eq:app-diffusion}
\end{align}
For every $\varepsilon>0$, the absolute values of the last two terms
are bounded  by
\begin{equation}
 \varepsilon\|\nabla\delta_{\kappa} \BU\|_{L^2(B_b^+)}^2
       +C_\varepsilon(b-a)^{-2}\mathcal N_0^2.
 \label{eq:app-bulk-errors}
\end{equation}
Here we used \eqref{eq:app-correction-bound} and
$\|\delta_{\kappa} \BU\|_{L^2(B_b^+)}\le C\|\BU\|_{H^1(B_1^+)}$.
In particular, the cut-off factor remains $(b-a)^{-2}$ until the
iteration has been completed.

The force term satisfies 
\begin{align}
 \left|\int_{B_1^+}\BF\cdot\delta_{-\kappa}\BPhi_{\kappa}\,ds\right|
 &\le C\|\BF\|_{L^2(B_1^+)}\|\nabla\BPhi_{\kappa}\|_{L^2(B_b^+)}\notag\\
 &\le\varepsilon\|\nabla\delta_{\kappa} \BU\|_{L^2(B_b^+)}^2
       +C_\varepsilon(b-a)^{-2}\mathcal N_0^2.
 \label{eq:app-force}
\end{align}
For the boundary term, the difference-quotient and trace bounds give
\begin{align}
 \|\delta_{\kappa} Y\|_{H^{-1/2}(\mathbb R)}
 &\le C\|Y\|_{H^{1/2}(I_1)},\notag\\
 \|\chi^2\delta_{\kappa} U_1\|_{H^{1/2}(\mathbb R)}
 &\le C\left(\|\chi\nabla\delta_{\kappa} \BU\|_{L^2(B_b^+)}
       +(b-a)^{-1}\|\BU\|_{H^1(B_1^+)}\right),
 \label{eq:app-trace}
\end{align}
where the first line uses the fixed extension of $Y$ specified above.
Since $\Phi_{\kappa,1}=\chi^2\delta_{\kappa} U_1$ on $I_b$, we obtain
\begin{equation}
 2|\langle\delta_{\kappa} Y,\Phi_{\kappa,1}\rangle|
 \le\varepsilon\|\nabla\delta_{\kappa} \BU\|_{L^2(B_b^+)}^2
       +C_\varepsilon(b-a)^{-2}\mathcal N_0^2.
 \label{eq:app-boundary-data}
\end{equation}
Combining these estimates and discarding
\eqref{eq:app-friction-positive}, we find, for every fixed $0<\theta<1$,
\begin{equation}
 E_{\kappa}(a)\le\theta E_{\kappa}(b)+C_\theta(b-a)^{-2}\mathcal N_0^2,
 \qquad E_{\kappa}(t)=\|\nabla\delta_{\kappa} \BU\|_{L^2(B_t^+)}^2.
 \label{eq:app-iteration}
\end{equation}
All constants here are independent of $\lambda$.

To close the estimate, take $t_j=7/8-2^{-j}/8$ and $\theta=1/8$.
Iteration gives
\[
 E_{\kappa}(3/4)\le\theta^m E_{\kappa}(t_m)
       +C\mathcal N_0^2\sum_{j=0}^{m-1}(4\theta)^j.
\]
For fixed $\kappa\ne0$, $E_{\kappa}(t_m)\le C|\kappa|^{-2}\|\BU\|_{H^1(B_1^+)}^2$.
Letting $m\to\infty$ yields a bound independent of $\kappa$.
The difference-quotient characterization of weak derivatives implies
\begin{equation}
 \|\partial_1\nabla \BU\|_{L^2(B_{3/4}^+)}\le C\mathcal N_0.
 \label{eq:app-tangential}
\end{equation}

\medskip
\noindent\textit{Step 2. The tangential pressure derivative.}
Put $\omega=B_{2/3}^+$. Given $q\in L^2(\omega)$ with zero mean,
choose $\Bvphi\in H_0^1(\omega;\mathbb R^2)$ satisfying
\begin{equation}
 \operatorname{div}\Bvphi=q,\qquad
 \|\nabla\Bvphi\|_{L^2(\omega)}\le C\|q\|_{L^2(\omega)}.
 \label{eq:app-pressure-lifting}
\end{equation}
Extend $\Bvphi$ by zero in the upper half-plane. Its tangential
translates still have zero trace on $I_1$. Testing the full
velocity--pressure weak identity with $-\delta_{-{\ka}}\Bvphi$
eliminates both boundary terms and gives
\begin{equation}
 \int_\omega\delta_{\ka} P\,q\,ds
 =\int_\omega\nabla\delta_{\ka} \BU:\nabla\Bvphi\,ds
       +\int_{B_1^+}\BF\cdot\delta_{-\ka}\Bvphi\,ds.
 \label{eq:app-pressure-difference}
\end{equation}
We use the full weak identity because $\Bvphi$ need not be solenoidal.
The translated copies of $\omega$ lie in $B_{3/4}^+$.
Thus \eqref{eq:app-tangential} and \eqref{eq:app-pressure-lifting} imply
\begin{equation}
 \|\delta_{\ka} P-c_{\ka}\|_{L^2(\omega)}\le C\mathcal N_0,
 \qquad c_{\ka}=(\delta_{\ka} P)_\omega.
 \label{eq:app-pressure-oscillation}
\end{equation}
Only the $L^2$ function $\delta_{\ka} P$ has been used; no regularity of
$\partial_1P$ is assumed at this stage.

Fix $\psi\in C_c^\infty(\omega)$ with $\int_\omega\psi\,ds=1$.
To control the mean, discrete integration by parts yields
\begin{align}
 c_{\ka}&=\int_\omega\delta_{\ka} P\,\psi\,ds
       -\int_\omega(\delta_{\ka} P-c_{\ka})\psi\,ds\notag\\
 &=-\int_{B_1^+}\bigl(P-(P)_{B_1^+}\bigr)
                  \delta_{-\ka}\psi\,ds
       -\int_\omega(\delta_{\ka} P-c_{\ka})\psi\,ds,
 \label{eq:app-pressure-mean}
\end{align}
where  we have used $\int\delta_{-\ka}\psi=0$. Since $\psi$ is fixed,
\eqref{eq:app-pressure-oscillation} implies $|c_{\ka}|\le C\mathcal N$.
Consequently, $\|\delta_{\ka} P\|_{L^2(\omega)}\le C\mathcal N$
uniformly for sufficiently small $|\ka|$. Passing to the limit gives
\begin{equation}
 \|\partial_1P\|_{L^2(\omega)}\le C\mathcal N.
 \label{eq:app-pressure-tangential}
\end{equation}
The pressure norm on the right is exactly the lower-order term
retained in Lemma~\ref{lem:local-stokes}.

\medskip
\noindent\textit{Step 3. Normal derivatives and conclusion.}
The divergence and momentum equations in \eqref{eq:local-stokes}
give the distributional identities
\begin{align}
 \partial_2^2U_2&=\partial_2G-\partial_1\partial_2U_1,
       \label{eq:app-normal-divergence}\\
 \partial_2^2U_1&=-\partial_1^2U_1+\partial_1P-F_1,
       \label{eq:app-normal-velocity}\\
 \partial_2P&=F_2+\partial_1^2U_2+\partial_2^2U_2.
       \label{eq:app-normal-pressure}
\end{align}
In particular, the nonzero divergence datum contributes
$\partial_2G$ in \eqref{eq:app-normal-divergence}.
Every term on the right of these identities belongs to $L^2(\omega)$
by \eqref{eq:app-tangential}, \eqref{eq:app-pressure-tangential}
and the assumptions on $\BF$ and $G$. Hence
\[
 \|\nabla^2\BU\|_{L^2(\omega)}+\|\nabla P\|_{L^2(\omega)}
       \le C\mathcal N.
\]
Restricting to $B_{1/2}^+$ and adding the lower-order velocity norm
yields the estimate in Lemma~\ref{lem:local-stokes}.
Since $P\in L^2(B_1^+)$, this also gives $P\in H^1(B_{1/2}^+)$.
Finally, $\lambda$ occurs only in the nonnegative term
\eqref{eq:app-friction-positive}, and the pressure tests have zero
boundary trace. The constant is therefore independent of every
constant $\lambda\ge0$.
\end{proof}

%%%%%%%%%%%%%%%%%%%%%%%%%%%%%%%%%%%%%%%%%%%%%%%%%%%%%%%%%%%%%%%%%%%%
\medskip

{\bf Acknowledgements.} 
H. Li was supported by the Postgraduate Research \& Practice Innovation Program of Jiangsu Province (Grant No.~5832001826).  J. Han was supported by the start-up funds for doctoral research at Anhui Normal University (Grant No.~762610).
The authors thank Professor Yun Wang for helpful discussions and suggestions.

	\medskip
	{\bf Statements and declarations.} The authors have no relevant financial or non-financial interests to disclose. Data sharing is not applicable to this article as no datasets were generated or analyzed during the current study.

\end{document}